\UseRawInputEncoding
\documentclass[12pt, reqno]{amsart}
\usepackage[margin=1in]{geometry}
\usepackage{amssymb,latexsym,amsmath,amscd,amsfonts}
\usepackage{latexsym}
\usepackage[mathscr]{eucal}
\usepackage{bm}
\usepackage{mathptmx}
\usepackage{amssymb}
\usepackage{amsthm}
\usepackage{dcolumn}
\usepackage[all]{xy}
\usepackage{enumitem}
\usepackage[utf8]{inputenc}

\def\textmatrix#1&#2\\#3&#4\\{\bigl({#1 \atop #3}\ {#2 \atop #4}\bigr)}
\def\dispmatrix#1&#2\\#3&#4\\{\left({#1 \atop #3}\ {#2 \atop #4}\right)}
\newcommand{\beg}{\begin{equation}}
	\newcommand{\eeg}{\end{equation}}
\newcommand{\ben}{\begin{eqnarray*}}
	\newcommand{\een}{\end{eqnarray*}}

\newtheorem{thm}{Theorem}[section]

\newtheorem{lem}[thm]{Lemma}

\numberwithin{equation}{section} \theoremstyle{definition}
\newtheorem{defn}[thm]{Definition}

\newcommand{\C}{\mathbb{C}}

\newcommand{\T}{\mathbb{T}}

\def\textmatrix#1&#2\\#3&#4\\{\bigl({#1 \atop #3}\ {#2 \atop #4}\bigr)}
\def\dispmatrix#1&#2\\#3&#4\\{\left({#1 \atop #3}\ {#2 \atop #4}\right)}

\begin{document}
	\title[The $2 \times 2$ block $\mathbb A_r$-contractions]{The $2 \times 2$ block matrices associated with an annulus}
	\author[Pal and Tomar]{SOURAV PAL AND NITIN TOMAR}
	
	\address[Sourav Pal]{Mathematics Department, Indian Institute of Technology Bombay,
		Powai, Mumbai - 400076, India.} \email{sourav@math.iitb.ac.in}
	
	\address[Nitin Tomar]{Mathematics Department, Indian Institute of Technology Bombay, Powai, Mumbai-400076, India.} \email{tnitin@math.iitb.ac.in}		
	
	
	
	\thanks{The first named author is supported by the `Early Achiever Research Award' of IIT Bombay, India with Grant No. RI/0220-10001427-001. The second named author is supported by the Prime Minister's Research Fellowship (PMRF) with PMRF Id No. 1300140 of Govt. of India.}	
	

	\begin{abstract}
		A bounded Hilbert space operator $T$ for which the closure of the annulus 
		\begin{center}
			$\mathbb{A}_r=\{z \ : \ r<|z|<1\} \subseteq \mathbb{C}, \qquad (0<r<1)$
		\end{center}
		is a spectral set is called an $\mathbb A_r$-contraction. A celebrated theorem due to Douglas, Muhly and Pearcy gives a necessary and sufficient condition such that a $2 \times 2$ block matrix of operators
		$
		\begin{bmatrix}
		T_1 & X \\
		0 & T_2
		\end{bmatrix}
		$
		is a contraction. We seek an answer to the same question in the setting of annulus, i.e., under what conditions $\widetilde{T}_Y=\begin{bmatrix}
			T_1 & Y\\
			0 & T_2\\
		\end{bmatrix}
		$ becomes an $\mathbb A_r$-contraction. For a pair of $\mathbb A_r$-contractions $T_1,T_2$ and an operator $X$ that commutes with $T_1,T_2$, here we find a necessary and sufficient condition such that each of the block matrices
		\[
		T_X= \begin{bmatrix}
			T & X\\
			0 & T\\
		\end{bmatrix} \,, \quad \widehat{T}_X=\begin{bmatrix}
			T_1 & X(T_1-T_2)\\
			0 & T_2\\
		\end{bmatrix}
		\]
		becomes an $\mathbb A_r$-contraction. Thus, the general block matrix $\widetilde{T}_Y$ is an $\mathbb A_r$-contraction if $T_1-T_2$ (as in $\widehat{T}_X$) is invertible.
		
	\end{abstract}

	\maketitle	

	\section{Introduction}

	\noindent Throughout this paper, all operators are bounded linear operators on complex Hilbert spaces. For a complex Hilbert space $\mathcal{H}$, we denote by $\mathcal{B}(\mathcal{H})$ the space of all operators acting on $\mathcal{H}$. A contraction is an operator with norm not greater than $1$. For a contraction $T$, $D_T$ denotes the unique positive square root of $I-T^*T$. Also, $\mathbb D$ denotes the unit disk in the complex plane $\mathbb C$ with center at the origin.
	
	\smallskip
	
	 A compact set $K \subseteq \mathbb{C}$ is said to be a \textit{spectral set} for $T \in \mathcal{B}(\mathcal{H})$ if the spectrum $\sigma(T)$ of $T$ is contained in $K$ and von Neumann's inequality holds, i.e.  
	\begin{equation}\label{eqn1}
	\|f(T)\| \leq \underset{x \in K}{\sup}|f(x)|
	\end{equation}   
for all rational functions $f$ with poles off $K$. In \cite{vN}, von Neumann introduced the notion of spectral set and proved that an operator $T$ is a contraction if and only if the closed unit disk $\overline{\mathbb D}$ is a spectral set for $T$. This path breaking discovery motivates mathematicians to study operators having other domains as spectral sets and numerous interesting results are obtained, see \cite{paulsen} for the historical development. In this article, we consider an operator having a closed annulus as a spectral set.

\begin{defn}
An operator having the closure of the annulus
\[
\mathbb{A}_r=\{z \ : \ r<|z|<1\} \subseteq \mathbb{C}, \qquad (0<r<1)
\]
as a spectral set is called an $\mathbb A_r$-\textit{contraction}.
\end{defn}

Operator theory on an annulus has a rich literature, e.g. see \cite{Dmitry, DritschelI, Pas-McCull, comm lifting, Tsikalas}. Agler \cite{Agler} reduces $(\ref{eqn1})$ for the class of $\mathbb{A}_r$-contractions to some positivity condition of a certain family of operators. For $\epsilon$ in $(0,1)$, let us consider the function 
\[
\Gamma_\epsilon(z)=2\overset{\infty}{\underset{k=-\infty}{\sum}}\frac{(1-\epsilon)^k}{1+(1-\epsilon)^{2k}r^k} z^k \quad (r\leq |z| \leq 1). 
\]
It is not difficult to see that for small $\epsilon>0$, the series defining $\Gamma_\epsilon$ converges uniformly on compact subsets of $\{z \in \mathbb{C} \ : \ (1-\epsilon)r \leq |z| \leq 1\slash (1-\epsilon) \}.$ Therefore, if $T$ is an operator with $\sigma(T) \subseteq \overline{\mathbb{A}}_r,$ then $\Gamma_\epsilon(T)$ is well-defined. Agler proved the following result.
\begin{thm}[\cite{Agler}, Theorem 1.23]\label{Agler_pencil} A Hilbert space operator $T$ is an $\mathbb{A}_r$-contraction if and only if 
	\[
	\sigma(T) \subseteq \overline{\mathbb{A}}_r \quad \text{and} \quad \ Re \Gamma_\epsilon(\alpha T) \geq0
	\]
	for all $\epsilon \in (0,1)$ and $\alpha \in \mathbb{T}$, where $Re \Gamma_\epsilon(\alpha T)$ denotes the real part of the operator $ \Gamma_\epsilon(\alpha T)$.
\end{thm}
The aim of this short article is to investigate when a $2 \times 2$ block matrix of the form
$\begin{bmatrix}
		T_1 & Y \\
		0 & T_2
	\end{bmatrix}$
	becomes an $\mathbb A_r$-contraction. The motivation is to go parallel with the following famous result due to Douglas, Muhly and Pearcy.
	
\begin{thm}[\cite{Douglas}, Proposition 2.2]\label{Douglas}
	Let $T_1, T_2$ and $X$ be operators acting on a Hilbert space $\mathcal{H}$. Then $\widetilde{T}_X=\begin{bmatrix}
		T_1 & X \\
		0 & T_2
	\end{bmatrix}
	$
acting on $\mathcal{H}\oplus \mathcal{H}$ is a contraction if and only if $T_1$ and $T_2$ are contractions, and there is a contraction $C$ on $\mathcal{H}$ such that $X=D_{T_1^*}CD_{T_2}$. 
 \end{thm}
The first step in this direction for an annulus was taken by Misra \cite{Misra}, where he considered a $2 \times 2$ scalar matrix.
\begin{thm}[\cite{Misra}, Corollary 1.1]\label{Misra} Let $w \in \mathbb{A}_r$ and let $h(w)$ be a function of $w$. Then $ \overline{\mathbb{A}}_r$ is a spectral set for
	$\begin{pmatrix}
		w & h(w) \\
		0 & w 
	\end{pmatrix} $ if and only if 
$
	|h(w)| \leq \widehat{K}_r(\overline{w},w)^{-1},
$
where $\widehat{K}_r(z, w)$ is the reproducing kernel function of the Hardy space $H^2(\mathbb{A}_r)$ and is given by
$
	\widehat{K}_r(z, w)=\overset{\infty}{\underset{n=-\infty}{\sum}}\dfrac{(z\overline{w})^n}{1+r^{2n}}$ for $z,w \in \mathbb A_r$ .
\end{thm}
In Theorems \ref{block_1} \& \ref{thm:main-2} of this paper, we find a necessary and sufficient condition such that each of the following two block matrices
\begin{equation} \label{eqn:main-1}
T_X=\begin{bmatrix}
	T & X\\
	0 & T\\
\end{bmatrix}, \quad \widehat{T}_X=\begin{bmatrix}
	T_1 & X(T_1-T_2)\\
	0 & T_2\\
\end{bmatrix}
\end{equation}
becomes an $\mathbb A_r$-contraction. Here we assume that $X$ commutes with $T,T_1$ and $T_2$. Obviously, at this point we do not have an answer for the general case $\begin{bmatrix}
	T_1 & Y\\
	0 & T_2\\
\end{bmatrix}
$. However, $\widetilde{T}_Y=\widehat{T}_X$ when the operator equation $Y=X(T_1-T_2)$ has a solution for $X$. Thus, the problem is resolved when $T_1-T_2$ is invertible.

\vspace{0.2cm}
	\section{The main theorems}
	
	\vspace{0.2cm}
	
\noindent We begin with a few preparatory results. Our aim is to study when $T_X$ and $\widehat{T}_X$ are $\mathbb A_r$-contractions, where $T_X$ and $\widehat{T}_X$ are as in (\ref{eqn:main-1}).	
	
	\begin{lem}\label{action}
			Let $T$ be an $\mathbb{A}_r$-contraction acting on a Hilbert space $\mathcal{H}$ and let $X \in \mathcal{B}(\mathcal{H})$ be such that $TX=XT$. Then 
	\[
	 f(T_X)=\begin{bmatrix}
			f(T) & Xf'(T)\\
			0 & f(T)\\
		\end{bmatrix}
		\]
	for every rational function $f$ with poles off $\overline{\mathbb A}_r$, where $T_X$ is as in $(\ref{eqn:main-1})$.	
	\end{lem}

	\begin{proof}
		Let $p(z)=a_0+a_1z+a_2z^2+\dotsc + a_nz^n$ be any polynomial in $\C[z]$. Then
		\begin{equation*}
			\begin{split}
				p(T_X)&=a_0\begin{bmatrix}
					I & 0\\
					0 & I\\
				\end{bmatrix} +
				a_1\begin{bmatrix}
					T & X\\
					0 & T\\
				\end{bmatrix}+
				a_2\begin{bmatrix}
					T^2 & 2XT\\
					0 & T^2\\
				\end{bmatrix}+
				\dotsc+
				a_n\begin{bmatrix}
					T^n & nXT^{n-1}\\
					0 & T^n\\
				\end{bmatrix}\\
				&=\begin{bmatrix}
					p(T) & Xp'(T)\\
					0 & p(T)\\
				\end{bmatrix}.
			\end{split}
		\end{equation*}
		Let $f=p\slash q$ be a rational function with poles off $\overline{\mathbb{A}}_r.$
		The holomorphic functional calculus \cite{TaylorI} yields that 
		\begin{equation*}
			\begin{split}
				f(T_X)=p(T_X)q(T_X)^{-1}&=\begin{bmatrix}
					p(T) & Xp'(T)\\
					0 & p(T)\\
				\end{bmatrix}\begin{bmatrix}
					q(T) & Xq'(T)\\
					0 & q(T)\\
				\end{bmatrix}^{-1}\\
				&=\begin{bmatrix}
					p(T) & Xp'(T)\\
					0 & p(T)\\
				\end{bmatrix}\begin{bmatrix}
					q(T)^{-1} & -Xq(T)^{-2}q'(T)\\
					0 & q(T)^{-1}\\
				\end{bmatrix}\\
				&=\begin{bmatrix}
					p(T)q(T)^{-1} & X\bigg(p'(T)q(T)^{-1}-p(T)q(T)^{-2}q'(T)\bigg)\\
					0 & p(T)q(T)^{-1}\\
				\end{bmatrix}\\
				&=\begin{bmatrix}
					f(T) & Xf'(T)\\
					0 & f(T)\\
				\end{bmatrix}.\\
			\end{split}
		\end{equation*}
This completes the proof.
	\end{proof}

In a similar fashion, we prove an analogous result for the block matrix $\widehat{T}_X$.
	
	\begin{lem}\label{action2}
		Let $T_1, T_2$ be $\mathbb{A}_r$-contractions and let $X$ be an operator which commutes with $T_1$ and $T_2$. Then 
		\[
		 f(\widehat{T}_X)=\begin{bmatrix}
			f(T_1) & X\left(f(T_1)-f(T_2)\right)\\
			0 & f(T_2)\\
		\end{bmatrix}
		\]
		for every rational function $f$ with poles off $\overline{\mathbb A}_r$, where $\widehat{T}_X$ is as in $(\ref{eqn:main-1})$.
	\end{lem}
	
	\begin{proof}
		For any $n \in \mathbb{N} \cup \{0\}$, we have  
		\[
		\widehat{T}_X^n=\begin{bmatrix}
			T_1^n & X(T_1^n-T_2^n)\\
			0 & T_2^n\\
		\end{bmatrix}. 
	\]
For any polynomial $p \in \C[z]$, we have 	
	\[
		p(\widehat{T}_X)=\begin{bmatrix}
		p(T_1) & X(p(T_1)-p(T_2))\\
		0 & p(T_2)\\
	\end{bmatrix}.
	\]
A straight-forward computation shows that 
\[
p(\widehat{T}_X)^{-1}=\begin{bmatrix}
	p(T_1)^{-1} & -p(T_1)^{-1}X(p(T_1)-p(T_2))p(T_2)^{-1}\\
	0 & p(T_2)^{-1}\\
\end{bmatrix}.
\]
Since $X$ and $T$ commutes, $p(\widehat{T}_X)^{-1}$ can be re-written as 
\[
p(\widehat{T}_X)^{-1}=\begin{bmatrix}
	p(T_1)^{-1} & X(p(T_1)^{-1}-p(T_2)^{-1})\\
	0 & p(T_2)^{-1}\\
\end{bmatrix}.
\]
Let $f=p\slash q$ be a rational function with poles off $\overline{\mathbb{A}}_r$. then we have
 
\begin{equation*}
	\begin{split}
				 f(\widehat{T}_X)=p(\widehat{T}_X)q(\widehat{T}_X)^{-1}
		 &=\begin{bmatrix}
		 	p(T_1) & X(p(T_1)-p(T_2))\\
		 	0 & p(T_2)\\
		 \end{bmatrix}
	 \begin{bmatrix}
		 q(T_1)^{-1} & X(q(T_1)^{-1}-q(T_2)^{-1})\\
		 0 & q(T_2)^{-1}\\
	 \end{bmatrix}\\
	 &=\begin{bmatrix}
	 	p(T_1)q(T_1)^{-1} & X(p(T_1)q(T_1)^{-1}-p(T_1)q(T_2)^{-1})\\
	 	0 & p(T_2)q(T_2)^{-1}\\
	 \end{bmatrix}\\
	& =\begin{bmatrix}
	f(T_1) & X\left(f(T_1)-f(T_2)\right)\\
	0 & f(T_2)\\
\end{bmatrix}
\end{split}
\end{equation*} 
and the proof is complete.

	\end{proof}

	\noindent Also, the following popular result from the literature will be useful in the proof of the main theorems.    
	\begin{thm}[\cite{Sal}, Theorem 5.11]\label{Ando} Let $P$ and $Q$ be two positive operators on a Hilbert space $\mathcal{H}$ and $R \in \mathcal{B}(\mathcal{H})$. Then the following are equivalent.
		\begin{enumerate}
			\item The block matrix $\begin{bmatrix}
				P & R \\ R^* & Q \\
			\end{bmatrix}$ is positive;
			\item there exists a contraction operator $K$ on $\mathcal{H}$ such that $R = P^{1/2}KQ^{1/2}.$
		\end{enumerate}
	\end{thm}
	
	\noindent Now we are in a position to present the two main theorems of this article.

\begin{thm}\label{block_1}
		Let $T$ be an $\mathbb{A}_r$-contraction on a Hilbert space $\mathcal{H}$ and let $X \in \mathcal{B}(\mathcal{H})$ be such that $TX=XT$. Let $T_X$ acting on $\mathcal{H} \oplus \mathcal{H}$ be given by
		\[
		T_X=\begin{bmatrix}
			T & X\\
			0 & T\\
		\end{bmatrix}.
		\] 
		Then the following are equivalent.
		\begin{enumerate}
			\item $T_X$ is an $\mathbb{A}_r$-contraction;
			\item For every $\epsilon \in (0,1)$ and $\alpha \in\mathbb{T}$, there exists a contraction $K(\epsilon;\alpha)$ on $\mathcal{H}$ such that 
			\[
			\frac{X\Gamma'_\epsilon(\alpha T)}{2}=Re \Gamma_\epsilon(\alpha T)^{1\slash 2} K(\epsilon; \alpha) \ Re \Gamma_\epsilon(\alpha T)^{1\slash 2} \,;
			\]
			
			\item For every $\epsilon \in (0,1)$ and $\alpha \in\mathbb{T}$, there exists a  unitary $U(\epsilon;\alpha)$ on $\mathcal{H}\oplus \mathcal{H}$ such that 
			\[
			\frac{X\Gamma'_\epsilon(\alpha T)}{2}=\begin{bmatrix}
				Re \Gamma_\epsilon(\alpha T)^{1\slash 2} & 0\\
			\end{bmatrix}
			U(\epsilon; \alpha) \begin{bmatrix}
				Re \Gamma_\epsilon(\alpha T)^{1\slash 2}\\
				0\\
			\end{bmatrix}.
			\]
		\end{enumerate}
	Here, $Re \Gamma_\epsilon(\alpha T)^{1\slash 2}$ denotes the positive square root of $Re \Gamma_\epsilon(\alpha T)$ for $\epsilon \in (0,1)$ and $\alpha \in \T$.
	\end{thm}
	\begin{proof}
$(1) \iff (2)$.	Let $\epsilon \in (0,1)$ and $\alpha \in \mathbb{T}$. By Lemma \ref{action}, we have  
		\[
		\Gamma_\epsilon(\alpha T_X)=\begin{bmatrix}
			\Gamma_\epsilon(\alpha T) & X\Gamma'_\epsilon(\alpha T)\\
			0 & \Gamma_\epsilon(\alpha T)\\
		\end{bmatrix}
\, \, \text{and hence } \, \,		Re \Gamma_\epsilon(\alpha T_X)=
		\begin{bmatrix}
			Re \Gamma_\epsilon(\alpha T) & X\Gamma'_\epsilon(\alpha T)\slash 2\\
			(X \Gamma'_\epsilon(\alpha T))^*\slash 2 & Re \Gamma_\epsilon(\alpha T)\\
		\end{bmatrix}.
		\]

\vspace{0.3cm}
		
\noindent It follows from Theorem \ref{Agler_pencil} that $T_X$ is an $\mathbb{A}_r$-contraction if and only if $\sigma(T_X) \subseteq \overline{\mathbb{A}}_r$ and
	
\vspace{0.1cm}		
		\begin{equation}\label{eq:9.3.1}
		\begin{split}	
		\begin{bmatrix}
					Re \Gamma_\epsilon(\alpha T) & X\Gamma'_\epsilon(\alpha T)\slash 2\\
					(X \Gamma'_\epsilon(\alpha T))^*\slash 2 & Re \Gamma_\epsilon(\alpha T)\\
				\end{bmatrix}\geq 0 \ \ \mbox{for $\epsilon \in (0,1)$ and $\alpha \in \mathbb{T}.$}
			\end{split}
		\end{equation}
\vspace{0.1cm}	
	
	By Lemma 1 in \cite{Hong}, we have that $\sigma(T_X) \subseteq \sigma(T)\subseteq \overline{\mathbb{A}}_r$. Again by Theorem \ref{Agler_pencil}, $Re \Gamma_\epsilon(\alpha T) \geq 0$ for $\epsilon \in (0,1)$ and $\alpha \in \mathbb{T}$. It follows from Theorem \ref{Ando} that (\ref{eq:9.3.1}) holds if and only if there exists a contraction $K(\epsilon;\alpha)$ on $\mathcal{H}$ such that 
\vspace{0.1cm}
		\[
		\frac{X\Gamma'_\epsilon(\alpha T)}{2}=Re \Gamma_\epsilon(\alpha T)^{1\slash 2} K(\epsilon; \alpha) Re \Gamma_\epsilon(\alpha T)^{1\slash 2}
		\]
\vspace{0.1cm}
		for every $\epsilon \in (0,1)$ and $\alpha \in\mathbb{T}$.  
		
	\vspace{0.2cm}
		
\noindent $(2) \iff (3)$.	Let $U(\epsilon;\alpha)$ be the operator  on $\mathcal{H}\oplus \mathcal{H}$ given by
	\vspace{0.1cm}
		\[
		U(\epsilon;\alpha)=\begin{bmatrix}
			K(\epsilon;\alpha) & (I-K(\epsilon;\alpha)K(\epsilon;\alpha)^*)^{1\slash 2}\\
			(I-K(\epsilon;\alpha)^*K(\epsilon;\alpha))^{1\slash 2} & -K(\epsilon;\alpha)^*
		\end{bmatrix}.
		\]
	\vspace{0.1cm}	
		It is well-known that for any contraction $T$, the Halmos block $U=\begin{bmatrix}
	T & D_{T^*}\\
	D_T & -T^*
\end{bmatrix}$
is a unitary operator (see Chapter-I of \cite{Nagy} for details). Thus, it follows that each $U(\epsilon, \alpha)$ is a unitary. Now it is easy to see that 
		\[
\begin{bmatrix}
			Re \Gamma_\epsilon(\alpha T)^{1\slash 2} & 0\\
		\end{bmatrix}
		U(\epsilon; \alpha) \begin{bmatrix}
			Re \Gamma_\epsilon(\alpha T)^{1\slash 2}\\
			0\\
		\end{bmatrix}=Re \Gamma_\epsilon(\alpha T)^{1\slash 2} K(\epsilon; \alpha) Re \Gamma_\epsilon(\alpha T)^{1\slash 2}
		\]
		for every $\epsilon$ in $(0,1)$ and $\alpha \in \mathbb{T}.$ The desired conclusion now follows.
	\end{proof}

	\begin{thm} \label{thm:main-2}
		Let $T_1, T_2$ be $\mathbb{A}_r$-contractions on a Hilbert space $\mathcal{H}$ and let $X \in \mathcal{B}(\mathcal{H})$ commute with $T_1$ and $T_2$. Then, for the operator $\widehat{T}_X$ on $\mathcal{H}\oplus \mathcal{H}$ given by the $2 \times 2$ block matrix
		\[
		\widehat{T}_X=\begin{bmatrix}
			T_1 & X(T_1-T_2)\\
			0 & T_2\\
		\end{bmatrix},
		\]
		the following are equivalent:
		\begin{enumerate}
			\item $\widehat{T}_X$ is an $\mathbb{A}_r$-contraction;
			\item there exists a contraction operator $K(\epsilon;\alpha)$ on $\mathcal{H}$ such that 
			\[
			\frac{X\bigg(\Gamma_\epsilon(\alpha T_1)-\Gamma_\epsilon(\alpha T_2)\bigg)}{2}=Re \Gamma_\epsilon(\alpha T_1)^{1\slash 2} K(\epsilon; \alpha)\  Re \Gamma_\epsilon(\alpha T_2)^{1\slash 2}
			\]
			for every $\epsilon \in (0,1)$ and $\alpha \in\mathbb{T};$
			\item there exists a unitary $U(\epsilon;\alpha)$ on $\mathcal{H}\oplus \mathcal{H}$ such that 
			\[
			\frac{X\bigg(\Gamma_\epsilon(\alpha T_1)-\Gamma_\epsilon(\alpha T_2)\bigg)}{2}=\begin{bmatrix}
				Re \Gamma_\epsilon(\alpha T_1)^{1\slash 2} & 0\\
			\end{bmatrix}
			U(\epsilon; \alpha) \begin{bmatrix}
				Re \Gamma_\epsilon(\alpha T_2)^{1\slash 2}\\
				0\\
			\end{bmatrix}
			\]
			for every $\epsilon$ in $(0,1)$ and $\alpha \in \mathbb{T}.$
		\end{enumerate}
	\end{thm}
	\begin{proof}
$(1) \iff (2)$. By Lemma \ref{Agler_pencil} , $\widehat{T}_X$ is an $\mathbb{A}_r$-contraction if and only if  $\sigma(\widehat{T}_X) \subseteq \overline{\mathbb{A}}_r$ and
		\begin{equation}\label{eq:9.6.1}
			\begin{split}
 Re \Gamma_\epsilon(\alpha \widehat{T}_X)= \begin{bmatrix}
					Re \Gamma_\epsilon(\alpha T_1) & X\bigg(\Gamma_\epsilon(\alpha T_1)-\Gamma_\epsilon(\alpha T_2)\bigg)\slash 2\\
					X^*\bigg(\Gamma_\epsilon(\alpha T_1)^*-\Gamma_\epsilon(\alpha T_2)^*\bigg)\slash 2 & Re \Gamma_\epsilon(\alpha T_2)\\
				\end{bmatrix}\geq 0 
			\end{split}
		\end{equation}
		for every $\epsilon \in (0,1)$ and $\alpha \in \mathbb{T}$. Since $\widehat{T}_X=\begin{bmatrix}
			T_1 & X(T_1-T_2) \\
			0 & T_2
		\end{bmatrix}$ on $\mathcal{H}\oplus \mathcal{H}$ and $T_1, T_2$ are $\mathbb{A}_r$-contractions on $\mathcal{H}$, Lemma 1 in \cite{Hong} implies that $\sigma(T_X) \subseteq \sigma(T_1) \cup \sigma(T_2) \subseteq \overline{\mathbb{A}}_r$. By Theorem \ref{Agler_pencil}, we have that 
		\[
		Re \Gamma_\epsilon(\alpha T_1) \geq 0 \quad \mbox{and} \quad Re \Gamma_\epsilon(\alpha T_2) \geq 0 \quad \mbox{for every} \quad \epsilon \in (0,1) , \alpha \in \mathbb{T}.
		\]
		Now, it follows from Theorem \ref{Ando} that (\ref{eq:9.6.1}) holds if and only if there exists a contraction $K(\epsilon;\alpha)$ on $\mathcal{H}$ such that 
		\[
			\frac{1}{2}X\bigg(\Gamma_\epsilon(\alpha T_1)-\Gamma_\epsilon(\alpha T_2)\bigg)=Re \Gamma_\epsilon(\alpha T_1)^{1\slash 2} K(\epsilon; \alpha)\  Re \Gamma_\epsilon(\alpha T_2)^{1\slash 2}
		\]
		for every $\epsilon \in (0,1)$ and $\alpha \in\mathbb{T}$. 
		
	\vspace{0.2cm}
		
	\noindent $(2) \iff (3)$. For every $\epsilon$ in $(0,1)$ and $\alpha \in \mathbb{T}$, the operator $U(\epsilon;\alpha)$ given by
	\[
	U(\epsilon;\alpha)=\begin{bmatrix}
		K(\epsilon;\alpha) & (I-K(\epsilon;\alpha)K(\epsilon;\alpha)^*)^{1\slash 2}\\
		(I-K(\epsilon;\alpha)^*K(\epsilon;\alpha))^{1\slash 2} & -K(\epsilon;\alpha)^*
	\end{bmatrix}
	\]
	defines a unitary operator on $\mathcal{H} \oplus \mathcal{H}$. Again arguing as in the $(2) \Leftrightarrow (3)$ part of the proof of Theorem \ref{block_1}, we have that $	U(\epsilon;\alpha)$ is a unitary for every $\epsilon$ in $(0,1)$ and $\alpha \in \mathbb{T}$. A routine calculation yields
		\[
		\begin{bmatrix}
			Re \Gamma_\epsilon(\alpha T_1)^{1\slash 2} & 0\\
		\end{bmatrix}
		U(\epsilon; \alpha) \begin{bmatrix}
			Re \Gamma_\epsilon(\alpha T_2)^{1\slash 2}\\
			0\\
		\end{bmatrix}=Re \Gamma_\epsilon(\alpha T_1)^{1\slash 2} K(\epsilon; \alpha)\  Re \Gamma_\epsilon(\alpha T_2)^{1\slash 2}
		\]
		which shows the equivalence of $(2)$ and $(3)$. The proof is now complete.
	\end{proof}

\end{document}